\documentclass[reqno,b5paper]{amsart}

\usepackage{amsmath}
\usepackage{amssymb}
\usepackage{amsthm}
\usepackage{enumerate}
\usepackage[mathscr]{eucal}
\usepackage{eqlist}

\theoremstyle{plain}
\newtheorem{theorem}{Theorem}[section]

\newtheorem{lemma}[theorem]{Lemma}

\newtheorem{claim}{Claim}

\theoremstyle{definition}

\numberwithin{equation}{section}

\numberwithin{equation}{section}

\begin{document}
\title[Additivity of maps preserving Jordan $\eta_{\ast}$-products on $C^{*}$-algebras]%
{Additivity of maps preserving Jordan $\eta_{\ast}$-products on
$C^{*}$-algebras}
\author[A. Taghavi, H. Rohi and V. Darvish]%
{Ali Taghavi*, Hamid Rohi and Vahid Darvish}

\newcommand{\acr}{\newline\indent}
\address{\llap{*\,}Department of Mathematics,\\ Faculty of Mathematical
Sciences,\\ University of Mazandaran,\\ P. O. Box 47416-1468,\\
Babolsar, Iran.} \email{taghavi@umz.ac.ir, h.rohi@stu.umz.ac.ir,
v.darvish@stu.umz.ac.ir}

 \subjclass[2010]{47B48, 46L10}
\keywords{Maps preserving Jordan $\eta_{\ast}$-product, Additive,
Prime $C^{*}$-algebras}

\begin{abstract}
Let $\mathcal{A}$ and $\mathcal{B}$ be two $C^{*}$-algebras such
that $\mathcal{B}$ is prime. In this paper, we investigate the
additivity of map $\Phi$ from $\mathcal{A}$ onto $\mathcal{B}$ that
are bijective unital and satisfies
$$\Phi(AP+\eta PA^{*})=\Phi(A)\Phi(P)+\eta \Phi(P)\Phi(A)^{*},$$
for all $A\in\mathcal{A}$ and $P\in\{P_{1},I_{\mathcal{A}}-P_{1}\}$
where $P_{1}$ is a nontrivial projection in $\mathcal{A}$. Let
$\eta$ be a non-zero complex number such that $|\eta|\neq1$, then
$\Phi$ is additive. Moreover, if $\eta$ is rational then $\Phi$ is
$\ast$-additive.
\end{abstract}

\maketitle

\section{Introduction}\label{intro}
Let $\mathcal{R}$ and $\mathcal{R^{'}}$ be rings. We say the map
$\Phi: \mathcal{R}\to \mathcal{R^{'}}$ preserves product or is
multiplicative if $\Phi(AB)=\Phi(A)\Phi(B)$ for all $A, B\in
\mathcal{R}$. The question of when a product preserving or
multiplicative map is additive was discussed by several authors, see
\cite{mar} and references therein. Motivated by this, many authors
pay more attention to the map on rings (and algebras) preserving the
Lie product $[A,B]=AB-BA$ or the Jordan product $A\circ B=AB+BA$
(for example, Refs.
ref~\cite{bai,bai2,bre,hak,ji,lu,lu2,lu3,mol2,qi}). These results
show that, in some sense, the Jordan product or Lie product
structure is enough to determine the ring or algebraic structure.
Historically, many mathematicians devoted themselves to the study of
additive or linear Jordan or Lie product preservers between rings or
operator algebras. Such maps are always called Jordan homomorphism
or Lie homomorphism. Here we only list several results
\cite{bei,bei2,bei3,her,jac,kad,mar,mir,mir2}.

Let $\mathcal{R}$ be a $\ast$-ring and $\eta$ be a non-zero complex
scalar. For $A,B\in\mathcal{R}$, denoted the Jordan
$\eta_{\ast}$-product of $A$ and $B$ as $A\bullet_{\eta}B=AB+\eta
BA^{*}$ and Lie $\eta_{\ast}$-product as
$[A,B]_{\ast}^{\eta}=AB-\eta BA^{\ast}$. Particulary $A\bullet_{1}
B=AB+BA^{*}$ and $[A,B]_{\ast}^{1}=AB-BA^{*}$ are Jordan
$1_{\ast}$-product and Lie $1_{\ast}$-product, respectively. These
products are found playing a more and more important role in some
research topics, and its study has recently attracted many author's
attention (for example, see \cite{bre2,mol,sem}). A natural problem
is to study whether the map $\Phi$ preserving Jordan
$\eta_{\ast}$-product or Lie $\eta_{\ast}$-product on ring or
algebra $\mathcal{R}$ is a ring or algebraic isomorphism. In
\cite{cui}, J. Cui and C. K. Li proved a bijective map $\Phi$ on
factor von Neumann algebras which preserves the Lie
$1_{\ast}$-product ($[A,B]_{*}^{1}$) must be a $\ast$-isomorphism.
Moreover, in \cite{li} C. Li et al, discussed the nonlinear
bijective mapping preserving Jordan $1_{\ast}$-product
($A\bullet_{1} B$). They proved that such mapping on factor von
Neumann algebras is also $\ast$-ring isomorphism. These two articles
discussed new products for arbitrary operators on factor von Neumann
algebras. Also, A. Taghavi et al \cite{taghavi}, proved a bijective
unital map (not necessarily linear) on prime $C^{*}$-algebra which
preserving both Lie $1_{\ast}$-product and Jordan $1_{\ast}$-product
for which one of operators is projection must be $\ast$-additive
(i.e., additive and star-preserving). In a recent paper \cite{dai},
L. Dai and F. Lu proved a bijective map $\Phi$ on von Neumann
algebras which preserving Jordan $\eta_{\ast}$-product is a linear
$\ast$-isomorphism if $\eta$ is not real and $\Phi$ is a sum of a
linear $\ast$-isomorphism and a
conjugate linear $\ast$-isomorphism if $\eta$ is real.\\
In this paper, we will discuss such a bijective unital map on prime
$C^{*}$-algebra which preserving Jordan $\eta_{\ast}$-product
$\Phi(A\bullet_{\eta}P)=\Phi(A)\bullet_{\eta} \Phi(P)$ is additive.
We denote real part of an operator $T$ by $\Re(T)$, i.e.,
$\Re(T)=\frac{T+T^{*}}{2}$. It is well known that $C^{*}$-algebra
$\mathcal{A}$ is prime, in the sense that $A\mathcal{A}B=0$ for
$A,B\in \mathcal{A}$ implies either $A=0$ or $B=0$.

\section{Main Results}
We need the following lemmas for proving our Main Theorem.

\begin{lemma}
Let $\mathcal{A}$ and $\mathcal{B}$ be two $C^{*}$-algebras and
$\Phi:\mathcal{A}\to\mathcal{B}$ be a map which satisfies
$\Phi(A\bullet_{\eta}P)= \Phi(A)\bullet_{\eta}\Phi(P)$ for all
$A\in\mathcal{A}$ and some $P\in\mathcal{A}$ where $\eta$ is a
non-zero complex number such that $|\eta|\neq1$. Let $A, B$ and $T$
be in $\mathcal{A}$ such that $\Phi(T)=\Phi(A)+\Phi(B)$. Then we
have
\begin{equation}\label{standard}
\Phi(T\bullet_{\eta}P)=\Phi(A\bullet_{\eta}P)+\Phi(B\bullet_{\eta}P).
\end{equation}
\end{lemma}

\begin{proof}
Multiply the equalities $\Phi(T)=\Phi(A)+\Phi(B)$ and
$\Phi(T)^{*}=\Phi(A)^{*}+\Phi(B)^{*}$ by $\Phi(P)$ from the right
and $\eta \Phi(P)$ from the left. We get
$$\Phi(T)\Phi(P)=\Phi(A)\Phi(P)+\Phi(B)\Phi(P),$$
and
$$\eta\Phi(P)\Phi(T)^{*}=\eta\Phi(P)\Phi(A)^{*}+\eta\Phi(P)\Phi(B)^{*}.$$
By adding two equations, we have
$$\Phi(T)\Phi(P)+\eta\Phi(P)\Phi(T)^{*}=\Phi(A)\Phi(P)+\eta\Phi(P)\Phi(A)^{*}+\Phi(B)\Phi(P)+\eta\Phi(P)\Phi(B)^{*}.$$
\end{proof}

\begin{lemma}\label{l3}
Let $\mathcal{A}$ be a $C^{*}$-algebra. Suppose $T\in\mathcal{A}$
and $\eta$ be a non-zero complex number such that $|\eta|\neq1$. If
$T+\eta T^{*}=0$ then $T=0$.
\end{lemma}
\begin{proof}
Let $T+\eta T^{*}=0$, we can also write $T^{*}+\bar{\eta}T=0$. By an
easy computation we have $(1-\bar{\eta}\eta)T=0$. Then, $T=0$.
\end{proof}
\begin{lemma}\label{0}
Let $\mathcal{A}$ and $\mathcal{B}$ be two $C^{*}$-algebras with
identities and $\Phi:\mathcal{A}\to\mathcal{B}$ be a map which
satisfies $\Phi(A\bullet_{\eta} P)=\Phi(A)\bullet_{\eta} \Phi(P)$
for all $A\in\mathcal{A}$ and  $P\in\{P_{1},I-P_{1}\}$ where $P_{1}$
is a nontrivial projection in $\mathcal{A}$ and $\eta$ is a non-zero
complex number such that $|\eta|\neq1$, then $\Phi(0)=0.$
\end{lemma}
\begin{proof}
Let $\Phi(T)=0$, we prove that $T=0$. For showing this, apply Lemma
\ref{standard} to $\Phi(T)=0$ for $P_{1}$ and $P_{2}$, we have
$$\Phi(T\bullet_{\eta}P_{1})=0,$$
and
$$\Phi(T\bullet_{\eta}P_{2})=0.$$
So, by injectivity of $\Phi$ we obtain
$T\bullet_{\eta}P_{1}=T\bullet_{\eta}P_{2}$. By definition of Jordan
$\eta_{\ast}$-product we have $TP_{1}+\eta P_{1}T^{*}=TP_{2}+\eta
P_{2}T^{*}$, i.e., $(TP_{1}-TP_{2})+\eta(TP_{1}-TP_{2})^{*}=0.$ Now,
by implying Lemma \ref{l3}, we obtain $TP_{1}=TP_{2}$, we multiply
the latter equation by $P_{2}$ from right side, it follows
$TP_{2}=0$. So, $TP_{1}=0$. Now, we put $I-P_{2}$ instead of $P_{1}$
in $TP_{1}=0$, then we obtain $T(I-P_{2})=0$. Therefore $T=0$, since
$TP_{2}=0$.
\end{proof}
\paragraph{}
 Our main theorem is as follows:
\\
\\
\textbf{Main Theorem.} Let $\mathcal{A}$ and $\mathcal{B}$ be two
$C^{*}$-algebras such that $\mathcal{B}$ is prime, with
$I_{\mathcal{A}}$ and $I_{\mathcal{B}}$ the identities of them,
respectively. If $\Phi : \mathcal {A}\to \mathcal{B}$ is a unital
bijective map which satisfies $\Phi(A\bullet_{\eta}P)=
\Phi(A)\bullet_{\eta}\Phi(P)$ for all $A\in\mathcal{A}$ and
$P\in\{P_{1},I_{\mathcal{A}}-P_{1}\}$ where $P_{1}$ is a nontrivial
projection in $\mathcal{A}$ and $\eta$ is a non-zero complex number
such that $|\eta|\neq1$. Then, $\Phi$ is additive. Moreover, if
$\eta$ is rational then $\Phi$ is $\ast$-additive.
\\
\\
\textit{Proof of Main Theorem.} Let $P_{2}=I_{\mathcal{A}}-P_{1}$.
Denote $\mathcal{A}_{ij}=P_{i}\mathcal{A}P_{j},\ i,j=1,2,$ then
$\mathcal{A}=\sum_{i,j=1}^{2}\mathcal{A}_{ij}$. For every
$A\in\mathcal{A}$ we may write $A=A_{11}+A_{12}+A_{21}+A_{22}$. In
all that follows, when we write $A_{ij}$, it indicates that
$A_{ij}\in\mathcal{A}_{ij}$.\\
For showing additivity of $\Phi$ on $\mathcal{A}$ we will use above
partition of $\mathcal{A}$ and give some claims that prove $\Phi$ is
additive on each $\mathcal{A}_{ij}, \ i,j=1,2$.
\begin{claim}\label{1122}
For every $A_{11}\in\mathcal{A}_{11}$ and  $D_{22}\in
\mathcal{A}_{22}$, we have
$$\Phi(A_{11}+D_{22})=\Phi(A_{11})+\Phi(D_{22}).$$
\end{claim}
Since $\Phi$ is surjective, we can find an element
$T=T_{11}+T_{12}+T_{21}+T_{22}\in\mathcal{A}$ such that
\begin{equation}\label{ff1}
\Phi(T)=\Phi(A_{11})+\Phi(D_{22}),
\end{equation}
we should show $T=A_{11}+D_{22}$. We apply Lemma \ref{standard} to
(\ref{ff1}) for $P_{1}$, then we can write
$$\Phi(T\bullet_{\eta}P_{1})=\Phi(A_{11}\bullet_{\eta}P_{1})+\Phi(D_{22}\bullet_{\eta}P_{1}),$$
so,
$$\Phi(T_{11}+T_{21}+\eta T_{11}^{*}+\eta T_{21}^{*})=\Phi(A_{11}+\eta A_{11}^{*}),$$ by injectivity of $\Phi$, we get
$T_{11}+T_{21}+\eta T_{11}^{*}+\eta T_{21}^{*}=A_{11}+\eta
A_{11}^{*}$.\\
Multiply latter equation by $P_{2}$ from right and left side
respectively we obtain $T_{21}=T_{21}^{*}=0$. So, we have
$T_{11}+\eta T_{11}^{*}=A_{11}+\eta A_{11}^{*}$, or
$(T_{11}-A_{11})+\eta(T_{11}-A_{11})^{*}=0$. By Lemma \ref{l3}, we
get $T_{11}=A_{11}$.\\
Similarly, we apply  Lemma \ref{standard} to (\ref{ff1}) for
$P_{2}$, we have
$$\Phi(T\bullet_{\eta}P_{2})=\Phi(A_{11}\bullet_{\eta}P_{2})+\Phi(D_{22}\bullet_{\eta}P_{2}),$$
so,
$$\Phi(T_{12}+T_{22}+\eta T_{12}^{*}+\eta T_{22}^{*})=\Phi(D_{22}+\eta D_{22}^{*}),$$ by injectivity of $\Phi$, we get
$T_{12}+T_{22}+\eta T_{12}^{*}+\eta T_{22}^{*}=D_{22}+\eta
D_{22}^{*}$.\\
Multiply latter equation by $P_{1}$ from right and left side
respectively we obtain $T_{12}=T_{12}^{*}=0$. So, we have
$T_{22}+\eta T_{22}^{*}=D_{22}+\eta D_{22}^{*}$, or
$(T_{22}-D_{22})+\eta(T_{22}-D_{22})^{*}=0$. By Lemma \ref{l3}, we
get $T_{22}=D_{22}$.\\
\begin{claim}\label{1221}
For every $B_{12}\in\mathcal{A}_{12}$, $C_{21}\in\mathcal{A}_{21}$,
we have
$$\Phi(B_{12}+C_{21})=\Phi(B_{12})+\Phi(C_{21}).$$
\end{claim}
Let $T=T_{11}+T_{12}+T_{21}+T_{22}\in\mathcal{A}$ be such that
\begin{equation}\label{b}
\Phi(T)=\Phi(B_{12})+\Phi(C_{21}).
\end{equation}
By applying Lemma \ref{standard} to (\ref{b}) for $P_{1}$, we have
\begin{eqnarray*}
\Phi(T\bullet_{\eta}P_{1})&=&\Phi(B_{12}\bullet_{\eta}P_{1})+\Phi(C_{21}\bullet_{\eta}P_{1})\\
&=&\Phi(C_{21}\bullet_{\eta}P_{1}).
\end{eqnarray*}
Thus, by injectivity of $\Phi$ we have
$T\bullet_{\eta}P_{1}=C_{21}\bullet_{\eta}P_{1}$. It follows that
$$T_{11}+T_{21}+\eta T_{11}^{*}+\eta T_{21}^{*}=C_{21}+\eta
C_{21}^{*}.$$ Multiply above equation by $P_{2}$ from left side we
have $T_{21}=C_{21}$ and $T_{11}=0$.\\
Similarly, we can obtain $T_{12}=B_{12}$ and $T_{22}=0$ by applying
Lemma \ref{standard} to (\ref{b}) for $P_{2}$.
\begin{claim}\label{1112}
For every $A_{11}\in\mathcal{A}_{11}$ and  $B_{12}\in
\mathcal{A}_{12}$, we have
$$\Phi(A_{11}+B_{12})=\Phi(A_{11})+\Phi(B_{12}).$$
\end{claim}
Since $\Phi$ is surjective, we can find an element
$T=T_{11}+T_{12}+T_{21}+T_{22}\in\mathcal{A}$ such that
\begin{equation}\label{sh1}
\Phi(T)=\Phi(A_{11})+\Phi(B_{12}),
\end{equation}
we should show $T=A_{11}+B_{12}$. We apply Lemma \ref{standard} to
(\ref{sh1}) for $P_{1}$, then we can write $T_{11}=A_{11}$ and $T_{21}=0$.\\
Similarly, we apply Lemma \ref{standard} to (\ref{sh1}) for $P_{2}$,
we will have $T_{12}=B_{12}$ and $T_{22}=0$. So,
$T=A_{11}+B_{12}$.\\
\\
Note that $\Phi(C_{21}+D_{22})=\Phi(C_{21})+\Phi(D_{22})$ where
$C_{21}\in\mathcal{A}_{21}$ and $D_{22}\in\mathcal{A}_{22}$  can be
obtained as above.\\

\begin{claim}\label{1121}
For every $A_{11}\in\mathcal{A}_{11}$ and  $C_{21}\in
\mathcal{A}_{21}$, we have
$$\Phi(A_{11}+C_{21})=\Phi(A_{11})+\Phi(C_{21}).$$
\end{claim}
Since $\Phi$ is surjective, we can find an element
$T=T_{11}+T_{12}+T_{21}+T_{22}\in\mathcal{A}$ such that
\begin{equation}\label{f1}
\Phi(T)=\Phi(A_{11})+\Phi(C_{21}),
\end{equation}
we should show $T=A_{11}+C_{12}$. We apply Lemma \ref{standard} to
(\ref{f1}) for $P_{2}$, then we have $\Phi(T\bullet_{\eta}P_{2})=0$.
It means $\Phi(T_{12}+T_{22}+\eta T_{12}^{*}+\eta T_{22}^{*})=0$. By
Lemma \ref{0} and
Lemma \ref{l3} we obtain $T_{12}=T_{22}=0$.\\
On the other hand, We apply Lemma \ref{standard} to (\ref{f1}) for
$P_{1}$, then by Claim \ref{1112}, we have
\begin{eqnarray*}
\Phi(T\bullet_{\eta}P_{1})&=&\Phi(A_{11}\bullet_{\eta}P_{1})+\Phi(C_{21}\bullet_{\eta}P_{1})\\
&=&\Phi(A_{11}+\eta A_{11}^{*})+\Phi(C_{21}+\eta C_{21}^{*})\\
&=&\Phi(A_{11}+\eta A_{11}^{*})+\Phi(2\Re(C_{21}^{*})+(\eta-1)C_{21}^{*})\\
&=&\Phi(A_{11}+\eta A_{11}^{*}+2\Re(C_{21}^{*})+(\eta-1)C_{21}^{*})
\end{eqnarray*}
Then we get following
$$\Phi(T_{11}+T_{21}+\eta T_{11}^{*}+T_{21}^{*})=\Phi(A_{11}+\eta
A_{11}^{*}+C_{21}+\eta C_{21}^{*}).$$ Now, we imply Lemma \ref{l3}
to obtain $T_{11}=A_{11}$ and $T_{21}=C_{21}$.\\
\\
Note that $\Phi(B_{12}+D_{22})=\Phi(B_{12})+\Phi(D_{22})$ where
$B_{12}\in\mathcal{A}_{12}$ and $D_{22}\in\mathcal{A}_{22}$  can be
obtained as above.

\begin{claim}\label{ijij}
For every $A_{ij}, B_{ij}\in \mathcal{A}_{ij}$ such that $1\leq
i\neq j\leq 2$, we have
$$\Phi(A_{ij}+B_{ij})=\Phi(A_{ij})+\Phi(B_{ij}).$$
\end{claim}
Let $T=T_{11}+T_{12}+T_{21}+T_{22}\in \mathcal{A}$ be such that
\begin{equation}\label{bbb2}
\Phi(T)=\Phi(A_{ij})+\Phi(B_{ij}).
\end{equation}
By applying Lemma \ref{standard} to (\ref{bbb2}) for $P_{i}$, we get
$$\Phi(TP_{i}+\eta P_{i}T^{*})=\Phi(A_{ij}P_{i}+\eta P_{i}A_{ij}^{*})+\Phi(B_{ij}P_{i}+\eta P_{i}B_{ij}^{*})=\Phi(0)=0,$$
therefore, $\Phi(T_{ii}+T_{ji}+\eta T_{ii}^{*}+\eta T_{ji}^{*})=0$.
So, by Lemma \ref{l3} we have  $T_{ii}=T_{ji}=0$.\\
On the other hand, we apply Lemma \ref{standard} to (\ref{bbb2}) for
$P_{j}$ again, by Claim \ref{1221}, it follows
\begin{eqnarray*}
\Phi(TP_{j}+\eta P_{j}T^{*})&=&\Phi(A_{ij}P_{j}+\eta P_{j}A_{ij}^{*})+\Phi(B_{ij}P_{j}+\eta P_{j}B_{ij}^{*})\\
&=&\Phi(A_{ij}+\eta A_{ij}^{*})+\Phi(B_{ij}+\eta B_{ij}^{*})\\
&=&\Phi(2\Re(A_{ij}^{*})+(\eta-1)A_{ij}^{*})+\Phi(2\eta \Re(B_{ij})+(1-\eta)B_{ij})\\
&=&\Phi(2\Re(A_{ij}^{*})+(\eta-1)A_{ij}^{*}+2\eta \Re(B_{ij})+(1-\eta)B_{ij})\\
&=&\Phi(A_{ij}+\eta A_{ij}^{*}+B_{ij}+\eta B_{ij}^{*}).
\end{eqnarray*}
So, $$T_{ij}+T_{jj}+\eta T_{ij}^{*}+\eta T_{jj}^{*}=A_{ij}+\eta
A_{ij}^{*}+B_{ij}+\eta B_{ij}^{*}.$$ By Lemma \ref{l3} we have
$T_{jj}=0$ and $T_{ij}=A_{ij}+B_{ij}$.

\begin{claim}\label{111221}
For every $A_{11}\in\mathcal{A}_{11}$, $B_{12}\in\mathcal{A}_{12}$,
$C_{21}\in\mathcal{A}_{21}$, we have
$$\Phi(A_{11}+B_{12}+C_{21})=\Phi(A_{11})+\Phi(B_{12})+\Phi(C_{21}).$$
\end{claim}
Let $T=T_{11}+T_{12}+T_{21}+T_{22}\in\mathcal{A}$ be such that
\begin{equation}\label{b3}
\Phi(T)=\Phi(A_{11})+\Phi(B_{12})+\Phi(C_{21}).
\end{equation}
By applying Lemma \ref{standard} to (\ref{b3}) for $P_{2}$, we have
\begin{eqnarray*}
\Phi(T\bullet_{\eta}P_{2})&=&\Phi(A_{11}\bullet_{\eta}P_{2})+\Phi(B_{12}\bullet_{\eta}P_{2})+\Phi(C_{21}\bullet_{\eta}P_{2})\\
&=&\Phi(B_{12}\bullet_{\eta}P_{2}).
\end{eqnarray*}
Thus, by injectivity of $\Phi$ we have
$T\bullet_{\eta}P_{2}=B_{12}\bullet_{\eta}P_{2}$. It follows that
$$T_{22}+T_{12}+\eta T_{22}^{*}+\eta T_{12}^{*}=B_{12}+\eta
B_{12}^{*}.$$ Multiply above equation by $P_{1}$ from left side we
have $T_{12}=B_{12}$ and $T_{22}=0$.\\
Similarly, we apply Lemma \ref{standard} to (\ref{b3}) for $P_{1}$.
By Claim \ref{1112}, we have the following
\begin{eqnarray*}
\Phi(T\bullet_{\eta}P_{1})&=&\Phi(A_{11}\bullet_{\eta}P_{1})+\Phi(B_{12}\bullet_{\eta}P_{1})+\Phi(C_{21}\bullet_{\eta}P_{1})\\
&=&\Phi(A_{11}\bullet_{\eta}P_{1})+\Phi(C_{21}\bullet_{\eta}P_{1})\\
&=&\Phi(A_{11}+\eta A_{11}^{*})+\Phi(C_{21}+\eta C_{21}^{*})\\
&=&\Phi(A_{11}+\eta
A_{11}^{*})+\Phi(2\Re(C_{21}^{*})+(\eta-1)C_{21}^{*})\\
&=&\Phi(A_{11}+\eta
A_{11}^{*}+2\Re(C_{21}^{*})+(\eta-1)C_{21}^{*})\\
&=&\Phi(A_{11}+\eta A_{11}^{*}+C_{21}+\eta C_{21}^{*})
\end{eqnarray*}
Injectivity of $\Phi$ implies $T_{11}+T_{21}+\eta T_{11}^{*}+\eta
T_{21}^{*}=A_{11}+C_{21}+\eta A_{11}^{*}+\eta C_{21}^{*}$. We have
$T_{11}=A_{11}$ and $T_{21}=C_{21}$.

\begin{claim}\label{11122122}
For every $A_{11}\in\mathcal{A}_{11}$, $B_{12}\in\mathcal{A}_{12}$,
$ {C_{21}}\in\mathcal{A}_{21}$ and $D_{22}\in\mathcal{A}_{22}$ we
have
$$\Phi(A_{11}+B_{12}+C_{21}+D_{22})=\Phi(A_{11})+\Phi(B_{12})+\Phi(C_{21})+\Phi(D_{22}).$$
\end{claim}
Assume $T=T_{11}+T_{12}+T_{21}+T_{22}$ which satisfies in
\begin{equation}\label{d1}
\Phi(T)=\Phi(A_{11})+\Phi(B_{12})+\Phi(C_{21})+\Phi(D_{22}).
\end{equation}
By using Lemma \ref{standard} to ($\ref{d1}$) for $P_{1}$, Claim
\ref{1221} and Claim \ref{111221}, we obtain
\begin{eqnarray*}
\Phi(T\bullet_{\eta}P_{1})&=&\Phi(A_{11}\bullet_{\eta}P_{1})+\Phi(B_{12}\bullet_{\eta}P_{1})+\Phi(C_{21}\bullet_{\eta}P_{1})+\Phi(D_{22}\bullet_{\eta}P_{1})\\
&=&\Phi(A_{11}\bullet_{\eta}P_{1})+\Phi(C_{21}\bullet_{\eta}P_{1})\\
&=&\Phi(A_{11}+\eta A_{11}^{*})+\Phi(C_{21}+\eta C_{21}^{*})\\
&=&\Phi(A_{11}+\eta A_{11}^{*})+\Phi(C_{21})+\Phi(\eta C_{21}^{*})\\
&=&\Phi(A_{11}+\eta A_{11}^{*}+C_{21}+\eta C_{21}^{*})
\end{eqnarray*}
Since $\Phi$ is injective we have  $T_{11}+T_{21}+\eta
T_{11}^{*}+\eta T_{21}^{*}=A_{11}+C_{21}+\eta A_{11}^{*}+\eta
C_{21}^{*}$. We obtain $T_{11}=A_{11}$ and $T_{21}=C_{21}$.\\
Similarly, apply Lemma \ref{standard} to ($\ref{d1}$) for $P_{2}$
and same computation as above we can easily obtain $T_{12}=B_{12}$
and $T_{22}=D_{22}$. So,
$\Phi(A_{11}+B_{12}+C_{21}+D_{22})=\Phi(A_{11})+\Phi(B_{12})+\Phi(C_{21})+\Phi(D_{22})$.
\begin{lemma}\label{ghab}
Let $\Phi$ satisfy the assumptions of the Main Theorem. Then, for
every $A\in\mathcal{A}$ we have the following
$$\Phi(A\bullet_{\eta}I)=\Phi(A)\bullet_{\eta}\Phi(I).$$
\end{lemma}
\begin{proof}
Definition of Jordan $\eta_{\ast}$-product implies that
\begin{equation}\label{vvv1}
\Phi(A\bullet_{\eta}P_{1})=\Phi(A)\bullet_{\eta}\Phi(P_{1}),
\end{equation}
and
\begin{equation}\label{vvv2}
\Phi(A\bullet_{\eta}P_{2})=\Phi(A)\bullet_{\eta}\Phi(P_{2}).
\end{equation}
Add Equations (\ref{vvv1}) and (\ref{vvv2}) together, we have
$$\Phi(A\bullet_{\eta}P_{1})+\Phi(A\bullet_{\eta}P_{2})=\Phi(A)\bullet_{\eta}(\Phi(P_{1})+\Phi(P_{2})).$$
By Claims \ref{1122}, \ref{ijij} and \ref{11122122}, we obtain
$$\Phi(A\bullet_{\eta}P_{1}+A\bullet_{\eta}P_{2})=\Phi(A)\bullet_{\eta}\Phi(P_{1}+P_{2}),$$
Equivalently,
$$\Phi(A\bullet_{\eta}(P_{1}+P_{2}))=\Phi(A)\bullet_{\eta}\Phi(P_{1}+P_{2}),$$
so, $\Phi(A\bullet_{\eta}I)=\Phi(A)\bullet_{\eta}\Phi(I).$
\end{proof}
\begin{lemma}\label{projection}
Let $\Phi$ satisfy the assumptions of the Main Theorem, then
$\Phi(P_{1})$ and $\Phi(P_{2})$ are  nontrivial orthogonal
projections in $\mathcal{B}$.
\end{lemma}
\begin{proof}
Let $P\in\{P_{1},P_{2}\}$, where $P_{i}$ for $1\leq i\leq2$ are
nontrivial projections in $\mathcal{A}$. By Lemma \ref{ghab} and
definition of Jordan $\eta_{\ast}$-product we have
\begin{eqnarray*}
\Phi(P\bullet_{\eta} I)&=&\Phi(P)\bullet_{\eta} \Phi(I)\\
\Phi(I\bullet_{\eta} P)&=&\Phi(I)\bullet_{\eta} \Phi(P).
\end{eqnarray*}
Since $\Phi$ is unital, above equations give us
\begin{eqnarray}
\Phi(P+\eta P)&=&\Phi(P)+\eta \Phi(P)^{*} \nonumber\\
\Phi(P+\eta P)&=&\Phi(P)+\eta \Phi(P), \label{projp}
\end{eqnarray}
we obtain $\Phi(P)=\Phi(P)^{*}$.\\
On the other hand,
$\Phi(P\bullet_{\eta}P)=\Phi(P)\bullet_{\eta}\Phi(P)$ then
$\Phi(P+\eta P)=\Phi(P)^{2}+\eta\Phi(P)^{2}$. It follows by
(\ref{projp}) that
$\Phi(P)+\eta\Phi(P)=\Phi(P)^{2}+\eta\Phi(P)^{2}$. Now, Lemma
\ref{l3} implies $\Phi(P)^{2}=\Phi(P)$.\\
We show that $\Phi(P_{1})$ and $\Phi(P_{2})$ are orthogonal. Let
$\Phi(P_{1})=Q_{1}$ and $\Phi(P_{2})=Q_{2}$, then by Claim
\ref{1122}, $Q_{1}+Q_{2}=I$. Also,
$Q_{1}.Q_{2}=\Phi(P_{1}).\Phi(P_{2})=\Phi(P_{1}).(\Phi(I)-\Phi(P_{1}))=0$.
\end{proof}

\begin{lemma}\label{p}
Let $\Phi$ satisfy the assumptions of the Main Theorem, we have
\begin{equation}\label{eq1}
\Phi(AP_{i})=\Phi(A)\Phi(P_{i}),
\end{equation}
for $1\leq i\leq 2$.
\end{lemma}
\begin{proof}
It is easy to check
$\Phi(AP_{i}\bullet_{\eta}I)=\Phi(A\bullet_{\eta}P_{i})$.\\
Above equation can be written as
$\Phi(AP_{i})+\eta\Phi(AP_{i})^{*}=\Phi(A)\Phi(P_{i})+\eta\Phi(P_{i})\Phi(A)^{*}$.
Equivalently,
$$(\Phi(AP_{i})-\Phi(A)\Phi(P_{i}))+\eta(\Phi(AP_{i})-\Phi(A)\Phi(P_{i}))^{*}=0,$$
By applying Lemma \ref{l3}, we have
$\Phi(AP_{i})=\Phi(A)\Phi(P_{i}).$
\end{proof}

Now, Claim \ref{projection} ensures that there exist nontrivial
projections $Q_{i}$ $(i=1,2)$ such that $\Phi(P_{i})=Q_{i}$ and
$Q_{1}+Q_{2}=I$. We can write
$\mathcal{B}=\sum_{i,j=1}^{2}\mathcal{B}_{ij}$ where
$\mathcal{B}_{ij}=Q_{i}\mathcal{B}Q_{j},\ i,j=1,2$.\\

We imply the primeness property of $\mathcal{B}$ just in this claim.
\begin{claim}\label{iiii}
For every $A_{ii}, B_{ii}\in \mathcal{A}_{ii}$, $1\leq i\leq 2$ we
have
$$\Phi(A_{ii}+B_{ii})=\Phi(A_{ii})+\Phi(B_{ii}).$$
\end{claim}
First, we will prove that
$\Phi(P_{i}A+P_{i}B)=\Phi(P_{i}A)+\Phi(P_{i}B)$ for every
$A,B\in \mathcal{A}$.\\
By Lemma \ref{p}, Claim \ref{ijij} and for every
$T_{ji}\in\mathcal{B}_{ji}$ such that $i\neq j$ we obtain
\begin{eqnarray*}
\left(\Phi(P_{i}A+P_{i}B)-\Phi(P_{i}A)-\Phi(P_{i}B)\right)T_{ji}&=&\left(\Phi(P_{i}A+P_{i}B)-\Phi(P_{i}A)-\Phi(P_{i}B)\right)Q_{j}TQ_{i}\\
&=&(\Phi(P_{i}A+P_{i}B)Q_{j}-\Phi(P_{i}A)Q_{j}\\
&&-\Phi(P_{i}B)Q_{j})TQ_{i}\\
&=&(\Phi(P_{i}AP_{j}+P_{i}BP_{j})-\Phi(P_{i}AP_{j})\\
&&-\Phi(P_{i}BP_{j}))TQ_{i}\\
&=&\left(\Phi(A_{ij}+B_{ij})-\Phi(A_{ij})-\Phi(B_{ij})\right)TQ_{i}\\
&=&\left(\Phi(A_{ij})+\Phi(B_{ij})-\Phi(A_{ij})-\Phi(B_{ij})\right)TQ_{i}\\
&=&0.
\end{eqnarray*}
By the primeness of $\mathcal{B}$, we have
\begin{equation}\label{eq5}
\Phi(P_{i}A+P_{i}B)=\Phi(P_{i}A)+\Phi(P_{i}A).
\end{equation}
 Now, multiply the
right side of equation (\ref{eq5}) by $\Phi(P_{i})=Q_{i}$ and use
Lemma \ref{p}, we obtain
$$\Phi(P_{i}AP_{i}+P_{i}BP_{i})=\Phi(P_{i}AP_{i})+\Phi(P_{i}BP_{i}).$$
\\
So, additivity of $\Phi$ comes from Claim \ref{ijij},
\ref{11122122}, \ref{iiii}. \\
It remains to prove $\Phi$ is $\ast$-preserving for non-zero
rational number $\eta$ such that $|\eta|\neq1$. Since $\Phi$ is
Jordan $\eta_{\ast}$-product preserving we have
$$\Phi(A\bullet_{\eta}I)=\Phi(A)\bullet_{\eta}\Phi(I),$$
$\Phi$ is unital, so
$$\Phi(A+\eta A^{*})=\Phi(A)+\eta \Phi(A)^{*}.$$
Additivity of $\Phi$ and above equation implies that
\begin{equation}\label{234}
\Phi(\eta A^{*})=\eta\Phi(A)^{*}.
\end{equation}
 Let $\eta=\frac{a}{b}$, where $a,
b$ are integers. It is easy to see that
\begin{equation}\label{star}
\Phi(\frac{1}{b}A^{*})=\frac{1}{b}\Phi(A^{*}).
\end{equation}
Now, by additivity of $\Phi$ and (\ref{star}) we have
$\Phi(\frac{a}{b}A^{*})=\frac{a}{b}\Phi(A^{*})$. It follows that
$\Phi(\eta A^{*})=\eta\Phi(A^{*})$. Hence, by latter equation and
equation (\ref{234}) we proved $\Phi(A^{*})=\Phi(A)^{*}$.

 \rightline{$\square$}





\end{document}